\documentclass[12pt]{article}

\usepackage[T1]{fontenc}
\usepackage[utf8]{inputenc}
\usepackage[english]{babel}

\usepackage{amsmath}
\usepackage{amssymb}
\usepackage{amsthm}

\usepackage[a4paper,margin=2.5cm]{geometry}
\usepackage{microtype}
\usepackage[hidelinks]{hyperref}

\hypersetup{
  pdftitle={On the Structure of Low-Dimensional Poisson Algebras over Arbitrary Fields},
  pdfauthor={A. V. Petrov; O. O. Pypka},
  pdfsubject={Structural classification of Poisson algebras of dimensions at most three over arbitrary fields},
  pdfkeywords={Poisson algebra, low-dimensional algebra, structural classification, commutative associative algebra, Lie algebra, arbitrary field, characteristic 2}
}

\usepackage{authblk}

\newtheorem{theorem}{Theorem}[section]
\newtheorem{proposition}[theorem]{Proposition}
\newtheorem{lemma}[theorem]{Lemma}
\newtheorem{corollary}[theorem]{Corollary}

\theoremstyle{definition}

\theoremstyle{remark}

\title{On the Structure of Low-Dimensional Poisson Algebras over Arbitrary Fields}

\author[1]{A.~V.~Petrov}
\author[1]{O.~O.~Pypka}

\affil[1]{Oles Honchar Dnipro National University,
Dnipro, Ukraine}

\date{}

\begin{document}

\maketitle

\begin{abstract}
We investigate the structure of Poisson algebras of dimensions at most
three over an arbitrary field. Our approach is based on the internal
structure of the associated commutative associative and Lie algebras,
with particular emphasis on the associative square $P^2$, the derived
Lie algebra $[P,P]$, the Lie center, the associative annihilator,
idempotents, ideals and decomposability.

We obtain a complete classification in dimensions one and two and give
a structural classification in dimension three. In dimension two, we
prove that the associative and Lie multiplications cannot be
simultaneously non-zero. In dimension three, the classification is
organized according to the dimension and position of the derived Lie
algebra and, in the case of trivial Lie multiplication, according to
the dimension of $P^2$.

The arbitrary-field setting leads to phenomena which do not occur over
the complex field. In particular, quadratic and cubic field extensions
appear naturally in the classification, some families depend on
equivalence classes of symmetric bilinear forms and on the structure of
three-dimensional Lie algebras over the ground field, and
characteristic $2$ gives an additional family of Poisson algebras with
both multiplications non-zero. Over the complex field, the resulting classification specializes,
up to changes of basis and notation, to the known classifications in
dimensions at most three.
\end{abstract}

\medskip


\medskip


\medskip

\section{Introduction}
\label{sec:introduction}

Poisson algebras combine two classical algebraic structures on the
same vector space: a commutative associative multiplication and a Lie
multiplication connected by the Leibniz identity. The interaction
between these two operations makes their structure substantially
richer than that of commutative associative algebras or Lie algebras
considered separately. Even in small dimensions, the classification
problem involves not only the possible associated algebraic
structures, but also the compatibility between them.

An important early systematic treatment of finite-dimensional Poisson
algebras was given by Goze and Remm in
\cite{GozeRemm2008}. They interpreted a Poisson algebra in terms of a
single non-associative multiplication and investigated algebraic,
cohomological and deformation-theoretic aspects of the resulting
structures. Although the general framework of
\cite{GozeRemm2008} is developed under restrictions on the
characteristic, its low-dimensional classification results concern
Poisson algebras over the complex field.

In dimension two, the complex classification is particularly simple.
If the associated Lie algebra is abelian, the problem reduces to the
classification of two-dimensional commutative associative algebras.
If the associated Lie algebra is non-abelian, the compatible
commutative associative multiplication is trivial. Thus, over
$\mathbb C$, there is no two-dimensional Poisson algebra for which
both multiplications are non-zero.

The three-dimensional complex case was also considered in
\cite{GozeRemm2008}. A later treatment was given by Abdelwahab,
Fernández Ouaridi and Martín González in
\cite{AbdelwahabFernandezMartin2025}. Their approach starts with a
fixed commutative associative algebra and studies Lie brackets
compatible with its multiplication. The automorphism group of the
associative algebra acts on the set of such brackets, and its orbits
determine the corresponding Poisson algebras up to isomorphism. This
led to a complete algebraic classification of three-dimensional
complex Poisson algebras, together with a correction to the earlier
classification. The same work also studies degenerations, orbit
closures and irreducible components of the corresponding algebraic
variety.

A different approach is used for nilpotent Poisson algebras in
\cite{AbdelwahabBarreiroCalderonFernandez2023}. Here nilpotency is
understood in the sense adopted in that paper for an algebra endowed
with the two Poisson operations. Using an analogue of the
Skjelbred--Sund method, the authors describe such algebras as central
extensions of lower-dimensional ones and obtain classifications in
dimensions at most four. The algebraic classification is carried out
over an arbitrary field of characteristic different from $2$, while
the geometric classification is considered for complex nilpotent
Poisson algebras of dimensions three and four. Thus a substantial part
of the low-dimensional arbitrary-field problem was already known
before the present work.

The relation with the nilpotent classification of
\cite{AbdelwahabBarreiroCalderonFernandez2023} is considered again
after the complete three-dimensional classification. In particular,
we identify explicitly which of the two- and three-dimensional
algebras obtained here are nilpotent in the sense used in that work.

More recently, the complete classification of four-dimensional
complex Poisson algebras was obtained by Abdelwahab and Sánchez in
\cite{AbdelwahabSanchez2025}. Their classification contains both
non-parametric and parametric families with non-zero commutative
associative multiplication, together with the four-dimensional complex
Lie algebras regarded as Poisson algebras with trivial associative
multiplication.

Structural properties of Poisson algebras have also been investigated
independently of low-dimensional classification. Lie solvability over
fields of arbitrary characteristic was studied by Siciliano and Usefi
in \cite{SicilianoUsefi2021}. Fernández Ouaridi, Navarro and Towers
considered abelian subalgebras and ideals of maximal dimension in
\cite{FernandezNavarroTowers2024}, while further structural properties
of nilpotent and solvable Poisson algebras were studied by Fernández
Ouaridi and Omirov in \cite{FernandezOmirov2025}.

The present paper adopts a structural approach rather than beginning
with a direct analysis of structure constants. A similar philosophy has proved
useful in the study of low-dimensional Leibniz algebras; see
\cite{KurdachenkoPypkaSubbotin2022}. In that setting, internal
invariants and characteristic substructures provide a natural route
to normal forms and isomorphism classes. Here we apply a related
viewpoint to Poisson algebras, where two interacting multiplications
must be treated simultaneously.

To the best of our knowledge, a complete classification up to
isomorphism of all two- and three-dimensional Poisson algebras over an
arbitrary ground field, with characteristic $2$ included, has not
previously been obtained in this form. The purpose of the present
paper is therefore to describe the structure of Poisson algebras of
dimensions at most three over an arbitrary field $F$ and to derive
their classification from this structural analysis. No restriction on
$\operatorname{char}(F)$ is imposed unless explicitly stated.

Our approach is based primarily on intrinsic subspaces and invariants
of the two associated algebras. In particular, the derived Lie algebra
\[
[P,P]
\]
and the associative square
\[
P^2
\]
play a central role. Their dimensions, together with the Lie center,
the associative annihilator, associative powers, ideals, idempotents
and direct-sum decompositions, provide effective tools for separating
the possible structures and determining the remaining isomorphism
parameters.

In dimension one, exactly two Poisson algebras occur, both with
trivial Lie multiplication. In dimension two, we prove that every
Poisson algebra over an arbitrary field has either trivial Lie
multiplication or trivial associative multiplication. Consequently,
there is no two-dimensional Poisson algebra with both multiplications
non-zero. The arbitrary-field classification is nevertheless richer
than the complex one, since quadratic field extensions occur
naturally among the associative cases.

The three-dimensional situation is substantially more involved. We
organize the classification according to
\[
\dim_F[P,P]=0,1,2,3.
\]
When $[P,P]=0$, the problem reduces to the structure of
three-dimensional commutative associative algebras and is further
organized according to $\dim_F P^2$. When
$\dim_F[P,P]=1$, two different situations arise according as the
derived algebra is contained in the Lie center or not. When
$\dim_F[P,P]=2$, the derived algebra is abelian and the Lie structure
is determined by an invertible linear operator on this
two-dimensional ideal. Finally, when $[P,P]=P$, the associated Lie
algebra is simple and the compatible associative multiplication is
necessarily trivial.

This yields a complete structural classification in dimension three.
Some structural entries represent individual isomorphism classes,
whereas others are naturally parameterized by field extensions,
matrix equivalence classes or isomorphism classes of
three-dimensional simple Lie algebras over $F$. This formulation is
particularly well suited to arbitrary ground fields and avoids
introducing artificial normal forms that would require algebraic
closure or restrictions on the characteristic.

Several field-dependent phenomena become visible in this setting.
Quadratic and cubic field extensions occur naturally, certain
families are controlled by equivalence classes of symmetric matrices,
and the Lie structures occurring in the cases
$\dim_F[P,P]=2$ and $[P,P]=P$ depend on the ground field. Moreover,
characteristic $2$ produces an additional family of genuinely mixed
Poisson algebras, that is, algebras for which both multiplications are
non-zero.

When $F=\mathbb C$, these field-dependent phenomena collapse. At the
end of Section~\ref{sec:dimension-three}, we give an explicit
correspondence between the structural types obtained here and the
known three-dimensional complex classification of
\cite{AbdelwahabFernandezMartin2025}. Thus the arbitrary-field
description specializes exactly, up to changes of basis and notation,
to the established complex classification.

The paper is organized as follows. Section~\ref{sec:preliminaries}
fixes the terminology and notation and records several elementary
structural lemmas. Sections~\ref{sec:dimension-one} and
\ref{sec:dimension-two} treat dimensions one and two, respectively.
Section~\ref{sec:dimension-three} is devoted to the
three-dimensional case and develops the classification according to
the structure of the derived Lie algebra, followed by comparisons
with the previously known nilpotent and complex classifications.

\section{Preliminaries}
\label{sec:preliminaries}

Throughout the paper, $F$ denotes an arbitrary field and all vector
spaces and algebras are finite-dimensional over $F$. No restriction on
$\operatorname{char}(F)$ is imposed unless explicitly stated.

A Poisson algebra over $F$ is an $F$-vector space $P$ equipped with
two $F$-bilinear operations: a commutative associative multiplication
\[
(x,y)\longmapsto xy
\]
and a Lie multiplication
\[
(x,y)\longmapsto [x,y],
\]
satisfying the Leibniz identity
\[
[xy,z]=x[y,z]+y[x,z]
\]
for all $x,y,z\in P$.

The Lie multiplication is understood to be alternating, that is,
\[
[x,x]=0
\]
for every $x\in P$, and satisfies the Jacobi identity
\[
[x,[y,z]]+[y,[z,x]]+[z,[x,y]]=0.
\]
In particular, bilinearity and alternation imply
\[
[x,y]=-[y,x].
\]
The use of alternation rather than skew-symmetry as a defining
condition is important in characteristic $2$.

The commutative associative multiplication is not assumed to have an
identity. If an identity exists, it will be specified explicitly.

We denote the associated commutative associative algebra by
\[
P(+,\cdot)
\]
and the associated Lie algebra by
\[
P(+,[\, ,\,]).
\]

An isomorphism of Poisson algebras is an $F$-linear bijection which
preserves both the associative multiplication and the Lie
multiplication.

A vector subspace $S$ of $P$ is called a \emph{Poisson subalgebra} if
\[
SS\subseteq S
\qquad\text{and}\qquad
[S,S]\subseteq S.
\]
A vector subspace $I$ of $P$ is called a \emph{Poisson ideal} if
\[
PI\subseteq I
\qquad\text{and}\qquad
[P,I]\subseteq I.
\]

For the commutative associative structure, we put
\[
P^2=\operatorname{span}_F\{xy\mid x,y\in P\}
\]
and define recursively
\[
P^{n+1}=P^nP
\qquad (n\geq2).
\]
Thus, throughout the paper, $P^n$ denotes an associative power.
The associated commutative associative algebra is called
\emph{nilpotent} if
\[
P^n=0
\]
for some positive integer $n$.

The \emph{associative annihilator} is
\[
\operatorname{Ann}(P(+,\cdot))
=
\{a\in P\mid aP=0\}.
\]

For the associated Lie algebra, we write
\[
[P,P]
=
\operatorname{span}_F\{[x,y]\mid x,y\in P\}
\]
for its derived algebra. Its center is denoted by
\[
\zeta(P(+,[\, ,\,]))
=
\{z\in P\mid [z,x]=0
\text{ for every }x\in P\}.
\]

When solvability of the associated Lie algebra
\[
L=P(+,[\, ,\,])
\]
is considered, its derived series is defined by
\[
L^{(0)}=L,\qquad
L^{(n+1)}=[L^{(n)},L^{(n)}].
\]
The Lie algebra $L$ is called \emph{solvable} if
\[
L^{(n)}=0
\]
for some $n$.

Its lower central series is defined by
\[
\gamma_1(L)=L,\qquad
\gamma_{n+1}(L)=[\gamma_n(L),L].
\]
The Lie algebra $L$ is called \emph{nilpotent} if
\[
\gamma_{c+1}(L)=0
\]
for some positive integer $c$. The least such $c$ is called the
nilpotency class of $L$.

We shall explicitly distinguish between nilpotency of the
commutative associative algebra $P(+,\cdot)$ and nilpotency of the
associated Lie algebra $P(+,[\, ,\,])$.

A Poisson algebra $P$ is called \emph{decomposable} if
\[
P=I\oplus J
\]
for some non-zero Poisson ideals $I$ and $J$. Otherwise $P$ is called
\emph{indecomposable}.

An element $e$ of a commutative associative algebra is called an
\emph{idempotent} if
\[
e^2=e.
\]
If the algebra has an identity $1$, an idempotent is called
\emph{non-trivial} if
\[
e\neq0,1.
\]

For a finite-dimensional commutative associative algebra $A$, we
denote by
\[
\operatorname{Nil}(A)
=
\{a\in A\mid a^n=0
\text{ for some positive integer }n\}
\]
its nilradical. Since $A$ is commutative,
$\operatorname{Nil}(A)$ is an ideal.

If $E/F$ is a finite field extension, its degree is denoted by
\[
[E:F].
\]
Field extensions appearing in the classification are always
considered up to $F$-algebra isomorphism.

Unless otherwise stated, all products and Lie brackets between basis
elements which are neither explicitly displayed nor determined from
the displayed relations by commutativity, bilinearity and
alternation are understood to be zero.

We record several elementary observations that will be used repeatedly
in the sequel.

\begin{lemma}
\label{lem:basic-poisson-ideals}
For every Poisson algebra $P$, the subspaces $P^2$ and
$\operatorname{Ann}(P(+,\cdot))$ are Poisson ideals of $P$.
\end{lemma}

\begin{proof}
Clearly,
\[
PP^2\subseteq P^2.
\]
For $x,y,z\in P$, the Leibniz identity gives
\[
[x,yz]=y[x,z]+z[x,y]\in P^2.
\]
Hence
\[
[P,P^2]\subseteq P^2,
\]
and therefore $P^2$ is a Poisson ideal.

Let
\[
a\in\operatorname{Ann}(P(+,\cdot)).
\]
For arbitrary $x,y\in P$, we have $ay=0$, and therefore
\[
0=[ay,x]=a[y,x]+y[a,x].
\]
The first term on the right-hand side is zero, since $aP=0$. Hence
\[
y[a,x]=0
\]
for every $y\in P$, and consequently
\[
[a,x]\in\operatorname{Ann}(P(+,\cdot)).
\]
Thus the associative annihilator is also a Poisson ideal.
\end{proof}

\begin{lemma}
\label{lem:idempotents-central}
Every idempotent of $P(+,\cdot)$ belongs to the center of
$P(+,[\, ,\,])$.
\end{lemma}

\begin{proof}
Let $e^2=e$ and let $x\in P$. Put
\[
d=[e,x].
\]
The Leibniz identity gives
\[
d=[e^2,x]=2e[e,x]=2ed.
\]
Multiplying by $e$ and using $e^2=e$, we obtain
\[
ed=2ed,
\]
and hence $ed=0$. Consequently
\[
d=2ed=0.
\]
Thus
\[
[e,x]=0
\]
for every $x\in P$, and therefore
\[
e\in\zeta(P(+,[\, ,\,])).
\]
In characteristic $2$, the equality $d=2ed$ already gives $d=0$.
\end{proof}

The following elementary fact will be useful whenever the associative
square coincides with the whole algebra.

\begin{lemma}
\label{lem:square-equals-algebra}
Let $A$ be a finite-dimensional commutative associative algebra over
$F$. If
\[
A^2=A,
\]
then $A$ has an identity element.
\end{lemma}

\begin{proof}
Choose a basis
\[
a_1,\ldots,a_n
\]
of $A$. Since $A=A^2$, for every $i$ there exist elements
$c_{ij}\in A$ such that
\[
a_i=\sum_{j=1}^n c_{ij}a_j.
\]
Put
\[
C=(c_{ij})
\]
and consider the unitization $F1\oplus A$. Then
\[
(I-C)
\begin{pmatrix}
a_1\\
\vdots\\
a_n
\end{pmatrix}
=0.
\]
Since the unitization is commutative, multiplication by the adjugate
matrix yields
\[
\det(I-C)a_i=0
\qquad (1\leq i\leq n).
\]
The constant term of $\det(I-C)$ is $1$, while every other term
belongs to $A$. Hence
\[
\det(I-C)=1-e
\]
for some $e\in A$. Therefore
\[
(1-e)a_i=0
\]
for every $i$, and consequently
\[
ea=a
\]
for every $a\in A$. Thus $e$ is an identity element of $A$.
\end{proof}

Combining the preceding two lemmas, we obtain an observation which
will be particularly useful below: if a finite-dimensional Poisson
algebra $P$ satisfies
\[
P^2=P,
\]
then $P(+,\cdot)$ has an identity element $e$, and necessarily
\[
e\in\zeta(P(+,[\, ,\,])).
\]

\section{Poisson algebras of dimension one}
\label{sec:dimension-one}

We begin with the one-dimensional case. The Lie structure is
necessarily trivial, so the classification is completely determined
by the commutative associative multiplication. No restriction on the
characteristic of the ground field is required.

\begin{theorem}
Let $P$ be a Poisson algebra of dimension $1$ over an arbitrary field
$F$. Then $P$ is isomorphic to exactly one of the following algebras:
\[
\begin{array}{ll}
P_{1,1}=Fa, & a^2=0,\quad [a,a]=0,\\[2mm]
P_{1,2}=Fe, & e^2=e,\quad [e,e]=0.
\end{array}
\]
The algebras $P_{1,1}$ and $P_{1,2}$ are not isomorphic.
\end{theorem}

\begin{proof}
Let $0\neq a\in P$. Since $\dim_F P=1$, we have $P=Fa$. The Lie
multiplication is alternating, so $[a,a]=0$, and hence
\[
[P,P]=\langle0\rangle.
\]

There exists $\alpha\in F$ such that
\[
a^2=\alpha a.
\]
If $\alpha=0$, then $P\cong P_{1,1}$. Suppose that $\alpha\neq0$ and
put
\[
e=\alpha^{-1}a.
\]
Then
\[
e^2=\alpha^{-2}a^2=\alpha^{-1}a=e,
\]
and therefore $P\cong P_{1,2}$. Since the Lie multiplication is zero,
the Leibniz identity is automatically satisfied in both cases.

Finally,
\[
P_{1,1}^{\,2}=\langle0\rangle,
\qquad
P_{1,2}^{\,2}=P_{1,2}.
\]
Thus the two algebras are not isomorphic.
\end{proof}

We record the internal structure of these algebras. In both cases the
associated Lie algebra is abelian and
\[
\zeta(P_{1,j}(+,[\, ,\,]))=P_{1,j}
\qquad (j=1,2).
\]
Consequently, both associated Lie algebras are nilpotent of class $1$
and solvable.

For the associative structures,
\[
P_{1,1}^{\,2}=\langle0\rangle,\qquad
\operatorname{Ann}(P_{1,1}(+,\cdot))=P_{1,1},
\]
whereas
\[
P_{1,2}^{\,2}=P_{1,2},\qquad
\operatorname{Ann}(P_{1,2}(+,\cdot))=\langle0\rangle.
\]
Thus $P_{1,1}(+,\cdot)$ is nilpotent and has no non-zero idempotents,
while $P_{1,2}(+,\cdot)$ has the multiplicative identity $e$ and is
isomorphic to $F$.

Since a one-dimensional vector space has no non-zero proper subspaces,
the only Poisson subalgebras and Poisson ideals of either algebra are
$\langle0\rangle$ and the whole algebra. In particular, both algebras
are indecomposable as direct sums of non-zero Poisson ideals.

Thus precisely two one-dimensional Poisson algebras occur over every
field, and their structure is independent of the characteristic of the
ground field.

\section{Poisson algebras of dimension two}
\label{sec:dimension-two}

We now turn to the two-dimensional case. Here the associated Lie
algebra can be either abelian or non-abelian. An important feature of
this dimension is that a non-zero Lie multiplication is incompatible
with a non-zero associative multiplication. Thus the two operations
cannot be simultaneously non-zero.

When the Lie multiplication is trivial, the problem becomes purely
associative. The natural invariant in this case is $\dim_F P^2$, which
can take the values $0$, $1$, or $2$. When $P^2=P$, the associated
commutative associative algebra is unital by
Lemma~\ref{lem:square-equals-algebra}.

\begin{theorem}
\label{thm:dimension-two}
Let $P$ be a Poisson algebra of dimension $2$ over an arbitrary field
$F$. Then $P$ is isomorphic to exactly one of the following algebras,
where in {\rm (vi)} the quadratic extension is taken up to
$F$-algebra isomorphism.

\begin{enumerate}
\renewcommand{\labelenumi}{(\roman{enumi})}

\item
$P_{2,1}=Fa\oplus Fb$, where all associative products and all Lie
brackets are zero.

\item
$P_{2,2}=Fa\oplus Fb$, where
\[
a^2=b,
\]
and all other associative products and all Lie brackets are zero.

\item
$P_{2,3}=Fe\oplus Fb$, where
\[
e^2=e,
\]
and all other associative products and all Lie brackets are zero.

\item
$P_{2,4}=Fe_1\oplus Fe_2$, where
\[
e_1^2=e_1,\qquad
e_2^2=e_2,\qquad
e_1e_2=0,
\]
and the Lie multiplication is zero.

\item
$P_{2,5}=Fe\oplus Fu$, where
\[
e^2=e,\qquad
eu=u,\qquad
u^2=0,
\]
and the Lie multiplication is zero.

\item
$P_{2,6}(E)=E$, where $E/F$ is a quadratic field extension, the
associative multiplication is the ordinary multiplication in $E$, and
the Lie multiplication is zero.

\item
$P_{2,7}=Fa\oplus Fb$, where
\[
[a,b]=b,
\]
and the associative multiplication is zero.

\end{enumerate}

Within {\rm (vi)}, the isomorphism type depends on the
$F$-algebra isomorphism class of $E$; with this understood, the listed
types are pairwise non-isomorphic.
\end{theorem}

\begin{proof}
Suppose first that the associated Lie algebra $P(+,[\, ,\,])$ is
non-abelian. Since the Lie multiplication is alternating and
$\dim_F P=2$, its derived algebra has dimension $1$. We may therefore
choose a basis $\{a,b\}$ such that
\[
[a,b]=b.
\]

Write the associative multiplication in the form
\[
a^2=\alpha a+\beta b,\qquad
ab=\gamma a+\delta b,\qquad
b^2=\varepsilon a+\eta b.
\]
Since
\[
[a^2,a]=0,
\]
we obtain $\beta=0$. Next, the Leibniz identity gives
\[
[ab,a]=a[b,a],
\]
and hence $\gamma=0$. Thus
\[
a^2=\alpha a,\qquad
ab=\delta b.
\]
Applying the Leibniz identity to $[ab,b]$, we obtain
\[
b^2=0.
\]

Finally,
\[
[a^2,b]=2a[a,b]
\]
gives
\[
\alpha=2\delta.
\]
On the other hand, associativity of
\[
(a^2)b=a(ab)
\]
gives
\[
\alpha\delta=\delta^2.
\]
Substituting $\alpha=2\delta$, we obtain
\[
\delta^2=0.
\]
Since $F$ is a field, $\delta=0$, and consequently $\alpha=0$.
Thus the associative multiplication is trivial and
\[
P\cong P_{2,7}.
\]
The argument remains valid in every characteristic.

Suppose now that $P(+,[\, ,\,])$ is abelian. Then the Lie
multiplication is zero, and it remains to classify the commutative
associative algebra $P(+,\cdot)$.

If
\[
P^2=\langle0\rangle,
\]
then all associative products vanish and
\[
P\cong P_{2,1}.
\]

Let
\[
\dim_F P^2=1.
\]
Put
\[
P^2=Fb
\]
and choose $a\notin Fb$. Then
\[
a^2=\alpha b,\qquad
ab=\beta b,\qquad
b^2=\gamma b.
\]
Associativity of
\[
(a^2)b=a(ab)
\]
gives
\[
\alpha\gamma=\beta^2.
\]

If $\gamma=0$, then $\beta=0$. Since $P^2\neq\langle0\rangle$, we
have $\alpha\neq0$. Replacing $b$ by $\alpha b$, we obtain
\[
a^2=b,
\]
and therefore
\[
P\cong P_{2,2}.
\]

Suppose that $\gamma\neq0$. Put
\[
e=\gamma^{-1}b.
\]
Then
\[
e^2=e,\qquad ae=\beta e.
\]
Moreover, $\alpha\gamma=\beta^2$ gives
\[
a^2=\beta^2e.
\]
For
\[
c=a-\beta e
\]
we therefore have
\[
ec=0,\qquad c^2=0.
\]
Hence
\[
P\cong P_{2,3}.
\]

It remains to consider
\[
P^2=P.
\]
By Lemma~\ref{lem:square-equals-algebra}, the commutative associative
algebra $P(+,\cdot)$ has an identity element $e$.

Choose $x\notin Fe$. Then
\[
P=Fe\oplus Fx
\]
and
\[
x^2=\alpha e+\beta x
\]
for suitable $\alpha,\beta\in F$. Consequently,
\[
P(+,\cdot)
\cong
F[t]/(t^2-\beta t-\alpha).
\]
Indeed, the homomorphism
\[
F[t]\longrightarrow P
\]
defined by
\[
1\longmapsto e,\qquad
t\longmapsto x
\]
is surjective, and both the indicated quotient and $P$ have dimension
$2$ over $F$.

Put
\[
f(t)=t^2-\beta t-\alpha.
\]
If $f$ has two distinct roots in $F$, then
\[
F[t]/(f(t))\cong F\times F,
\]
which gives $P_{2,4}$.

If
\[
f(t)=(t-\lambda)^2
\]
for some $\lambda\in F$, then
\[
u=x-\lambda e
\]
satisfies
\[
u^2=0,
\]
and we obtain $P_{2,5}$.

Finally, if $f$ is irreducible over $F$, then
\[
F[t]/(f(t))
\]
is a quadratic field extension of $F$, which gives $P_{2,6}(E)$.

All the algebras listed above clearly satisfy the Poisson identities.
It remains to distinguish their isomorphism types.

The algebra $P_{2,7}$ is the only one with non-zero Lie
multiplication. Among the remaining algebras,
\[
\dim_F P_{2,1}^{\,2}=0,
\]
\[
\dim_F P_{2,2}^{\,2}
=
\dim_F P_{2,3}^{\,2}
=1,
\]
while
\[
P_{2,j}^{\,2}=P_{2,j}
\qquad (j=4,5,6).
\]
The algebra $P_{2,2}$ has no non-zero idempotents, whereas
$P_{2,3}$ contains the idempotent $e$.

Finally, $P_{2,4}$ has two non-zero proper ideals, $P_{2,5}$ has a
unique non-zero proper ideal, and $P_{2,6}(E)$ has no non-zero proper
ideals. Hence no further isomorphisms are possible.
\end{proof}

\begin{corollary}
\label{cor:dimension-two-no-mixed}
There is no two-dimensional Poisson algebra over an arbitrary field
with both multiplications non-zero.
\end{corollary}

\begin{proof}
By Theorem~\ref{thm:dimension-two}, every two-dimensional Poisson
algebra has either trivial Lie multiplication or trivial associative
multiplication.
\end{proof}

We now record the internal structure of the algebras obtained above.
Their associative squares and annihilators are as follows:
\[
\begin{array}{c|c|c}
P & P^2 & \operatorname{Ann}(P(+,\cdot))\\
\hline
P_{2,1} & \langle0\rangle & P_{2,1}\\
P_{2,2} & Fb & Fb\\
P_{2,3} & Fe & Fb\\
P_{2,4} & P_{2,4} & \langle0\rangle\\
P_{2,5} & P_{2,5} & \langle0\rangle\\
P_{2,6}(E) & P_{2,6}(E) & \langle0\rangle\\
P_{2,7} & \langle0\rangle & P_{2,7}
\end{array}
\]

The ideal structure also separates the different types. Every
subspace of $P_{2,1}$ is a Poisson ideal. The algebra $P_{2,2}$ has
the unique non-zero proper Poisson ideal $Fb$ and is indecomposable.
The non-zero proper Poisson ideals of $P_{2,3}$ are $Fe$ and $Fb$,
and
\[
P_{2,3}=Fe\oplus Fb
\]
is a direct sum of Poisson ideals.

The algebra $P_{2,4}$ is isomorphic to $F\times F$ and decomposes as
\[
P_{2,4}=Fe_1\oplus Fe_2.
\]
The algebra $P_{2,5}$ has the unique non-zero proper Poisson ideal
$Fu$ and is indecomposable. Since $P_{2,6}(E)(+,\cdot)$ is a field,
$P_{2,6}(E)$ has no non-zero proper Poisson ideals.

For $P_{2,7}$, the subspace $Fb$ is the unique non-zero proper Poisson
ideal, and hence this algebra is indecomposable. Its associated Lie
algebra satisfies
\[
[P_{2,7},P_{2,7}]=Fb,\qquad
\zeta(P_{2,7}(+,[\, ,\,]))=\langle0\rangle.
\]
Moreover,
\[
[Fb,P_{2,7}]=Fb,
\qquad
[Fb,Fb]=\langle0\rangle.
\]
Consequently, the associated Lie algebra is solvable but not
nilpotent.

For $1\leq j\leq6$, the Lie multiplication of $P_{2,j}$ is zero.
Therefore
\[
[P_{2,j},P_{2,j}]=\langle0\rangle,\qquad
\zeta(P_{2,j}(+,[\, ,\,]))=P_{2,j},
\]
and all these associated Lie algebras are abelian.

When $F=\mathbb C$, the quadratic-extension case $P_{2,6}(E)$
disappears, and Theorem~\ref{thm:dimension-two} reduces to the known
classification of two-dimensional complex Poisson algebras.

\section{Three-dimensional Poisson algebras}
\label{sec:dimension-three}

We now turn to the three-dimensional case. In contrast with dimension
two, Poisson algebras with both multiplications non-zero already occur
in dimension three. Moreover, over an arbitrary field several
additional phenomena arise which disappear over algebraically closed
fields.

The classification will be organized primarily according to the
dimension of the derived algebra
\[
[P,P]
\]
of the associated Lie algebra. Thus we consider successively the cases
\[
\dim_F[P,P]=0,1,2,3.
\]
When the Lie multiplication is trivial, the classification is further
organized according to $\dim_F P^2$.

The elementary properties of $P^2$, the associative annihilator and
idempotents established in Section~\ref{sec:preliminaries} will be
used throughout without further mention.

\subsection{The case of trivial Lie multiplication}

We first consider
\[
[P,P]=\langle0\rangle.
\]
Then the Poisson structure is completely determined by the
three-dimensional commutative associative algebra $P(+,\cdot)$.
The natural invariant in this case is $\dim_F P^2$.

\begin{theorem}
\label{thm:3d-trivial-lie}
Let $P$ be a three-dimensional Poisson algebra over $F$ with
\[
[P,P]=\langle0\rangle.
\]
Then $P$ is isomorphic to exactly one of the following algebras,
subject only to the parameter identifications stated below. All
associative products not explicitly displayed are zero.

\begin{enumerate}
\item[\rm (i)]
$P_{3,1}$:
\[
P^2=\langle0\rangle.
\]

\item[\rm (ii)]
$P_{3,2}=Fe\oplus Fx\oplus Fy$:
\[
e^2=e.
\]

\item[\rm (iii)]
$P_{3,3}=Fx\oplus Fy\oplus Fz$:
\[
x^2=z.
\]

\item[\rm (iv)]
$P_{3,4}(M)=Fx\oplus Fy\oplus Fz$:
\[
x^2=\alpha z,\qquad
xy=\beta z,\qquad
y^2=\gamma z,
\]
where
\[
M=
\begin{pmatrix}
\alpha&\beta\\
\beta&\gamma
\end{pmatrix},
\qquad
\det M\neq0.
\]

\item[\rm (v)]
$P_{3,5}=Fx\oplus Fy\oplus Fz$:
\[
x^2=y,\qquad
xy=z.
\]

\item[\rm (vi)]
$P_{3,6}=Fe\oplus Fx\oplus Fy$:
\[
e^2=e,\qquad
x^2=y.
\]

\item[\rm (vii)]
$P_{3,7}=Fe_1\oplus Fe_2\oplus Fz$:
\[
e_1^2=e_1,\qquad
e_2^2=e_2.
\]

\item[\rm (viii)]
$P_{3,8}=Fe\oplus Fu\oplus Fz$:
\[
e^2=e,\qquad
eu=u,\qquad
u^2=0.
\]

\item[\rm (ix)]
$P_{3,9}(E)=E\oplus Fz$, where $E/F$ is a quadratic field
extension, the multiplication on $E$ is its ordinary field
multiplication, and
\[
zP=\langle0\rangle.
\]

\item[\rm (x)]
\[
P_{3,10}\cong F\times F\times F.
\]

\item[\rm (xi)]
\[
P_{3,11}\cong
F\times F[\varepsilon]/(\varepsilon^2).
\]

\item[\rm (xii)]
\[
P_{3,12}(E)\cong F\times E,
\]
where $E/F$ is a quadratic field extension.

\item[\rm (xiii)]
\[
P_{3,13}(K)=K,
\]
where $K/F$ is a cubic field extension.

\item[\rm (xiv)]
$P_{3,14}=Fe\oplus Fu\oplus Fv$:
\[
e^2=e,\qquad
eu=u,\qquad
ev=v,
\]
and
\[
u^2=uv=v^2=0.
\]

\item[\rm (xv)]
$P_{3,15}=Fe\oplus Fu\oplus Fv$:
\[
e^2=e,\qquad
eu=u,\qquad
ev=v,\qquad
u^2=v,
\]
with
\[
uv=v^2=0.
\]
Equivalently,
\[
P_{3,15}(+,\cdot)\cong F[t]/(t^3).
\]
\end{enumerate}

In case {\rm (iv)}, two matrices $M$ and $M'$ determine isomorphic
algebras if and only if
\[
M'=\lambda C^TMC
\]
for some $C\in GL_2(F)$ and $\lambda\in F^\times$.

In cases {\rm (ix)} and {\rm (xii)}, the parameter $E$ is taken up to
$F$-algebra isomorphism, and in case {\rm (xiii)} the cubic extension
$K/F$ is taken up to $F$-algebra isomorphism.
\end{theorem}

\begin{proof}
Put
\[
A=P(+,\cdot).
\]

Suppose first that
\[
A^2=\langle0\rangle.
\]
Then all products vanish and we obtain $P_{3,1}$.

Let
\[
\dim_F A^2=1
\]
and write
\[
A^2=Fz.
\]
Suppose first that
\[
zA\neq\langle0\rangle.
\]
There is a non-zero linear functional
\[
\varphi:A\longrightarrow F
\]
such that
\[
za=\varphi(a)z
\]
for every $a\in A$. Write
\[
ab=\lambda(a,b)z.
\]
Associativity gives
\[
(ab)z=a(bz),
\]
and therefore
\[
\lambda(a,b)\varphi(z)
=
\varphi(a)\varphi(b).
\]
Since $\varphi\neq0$, it follows that $\varphi(z)\neq0$.

Put
\[
e=\varphi(z)^{-1}z.
\]
Then
\[
e^2=e
\]
and, after the corresponding normalization of $\varphi$,
\[
ab=\varphi(a)\varphi(b)e.
\]
Thus
\[
\ker\varphi
\]
is a two-dimensional associative annihilator, and we obtain
$P_{3,2}$.

Suppose now that
\[
zA=\langle0\rangle.
\]
Then
\[
A^3=\langle0\rangle.
\]
Choose $x,y$ such that
\[
A=Fx\oplus Fy\oplus Fz.
\]
We have
\[
x^2=\alpha z,\qquad
xy=\beta z,\qquad
y^2=\gamma z.
\]
Thus the multiplication is determined by the non-zero symmetric matrix
\[
M=
\begin{pmatrix}
\alpha&\beta\\
\beta&\gamma
\end{pmatrix}.
\]

If
\[
\operatorname{rank}M=1,
\]
a suitable change of basis gives
\[
x^2=z,
\]
which yields $P_{3,3}$.

If
\[
\operatorname{rank}M=2,
\]
we obtain $P_{3,4}(M)$.

Let next
\[
\dim_F A^2=2.
\]

Assume first that $A$ contains a non-zero idempotent $e$. Multiplication
by $e$ defines an idempotent linear operator
\[
L_e:A\longrightarrow A,\qquad
L_e(a)=ea.
\]
Hence
\[
A=eA\oplus\ker L_e.
\]
Both summands are ideals, and their product is zero.

Since $\dim_F A=3$ and $\dim_F A^2=2$, we have
\[
\dim_F eA=1
\qquad\text{or}\qquad
\dim_F eA=2.
\]

Suppose first that
\[
\dim_F eA=1.
\]
Put
\[
B=\ker L_e.
\]
Then
\[
A=Fe\oplus B,\qquad
\dim_F B=2.
\]
Since
\[
A^2=Fe\oplus B^2
\]
and $\dim_F A^2=2$, we obtain
\[
\dim_F B^2=1.
\]
By the two-dimensional classification established in
Theorem~\ref{thm:dimension-two}, the algebra $B$ is either the algebra
with multiplication
\[
x^2=y
\]
or the direct sum of a one-dimensional idempotent algebra and a
one-dimensional zero algebra. These possibilities give $P_{3,6}$ and
$P_{3,7}$.

Suppose now that
\[
\dim_F eA=2.
\]
Then
\[
A=B\oplus Fz,
\]
where
\[
B=eA
\]
is a two-dimensional unital commutative associative algebra and
\[
Bz=\langle0\rangle.
\]
Necessarily
\[
z^2=0.
\]
Indeed, if $z^2\neq0$, then a non-zero scalar multiple of $z$ is an
idempotent. The sum of this idempotent and the identity of $B$ would
then be an identity of $A$, and consequently $A^2=A$, contrary to
$\dim_F A^2=2$.

By Theorem~\ref{thm:dimension-two}, the three possibilities for $B$
are
\[
F\times F,\qquad
F[\varepsilon]/(\varepsilon^2),
\qquad
E,\quad [E:F]=2.
\]
They yield $P_{3,7}$, $P_{3,8}$ and $P_{3,9}(E)$, respectively.

It remains to consider
\[
\dim_F A^2=2
\]
when $A$ has no non-zero idempotents. Put
\[
B=A^2.
\]
Then $B$ is a two-dimensional commutative associative algebra without
non-zero idempotents. By the two-dimensional classification, either
\[
B^2=\langle0\rangle
\]
or there is a basis $\{u,v\}$ of $B$ such that
\[
u^2=v,\qquad
uv=v^2=0.
\]

We show that the second possibility cannot occur. Choose
\[
x\notin B
\]
and consider the linear operator
\[
T:B\longrightarrow B,\qquad
T(b)=xb.
\]
Write
\[
T(u)=\alpha u+\beta v.
\]
Associativity gives
\[
T(v)=T(u^2)=T(u)u=\alpha v.
\]
If
\[
x^2=\gamma u+\delta v,
\]
then
\[
T^2(u)=x^2u.
\]
The left-hand side is
\[
T^2(u)
=
\alpha^2u+2\alpha\beta v,
\]
whereas
\[
x^2u=\gamma v.
\]
Thus
\[
\alpha=0,\qquad
\gamma=0.
\]
Consequently
\[
xB\subseteq Fv,\qquad
x^2\in Fv,\qquad
B^2=Fv,
\]
and hence
\[
A^2\subseteq Fv,
\]
contrary to $\dim_F A^2=2$.

Therefore
\[
B^2=\langle0\rangle.
\]
Again let
\[
T:B\longrightarrow B,\qquad T(b)=xb.
\]
Since $x^2\in B$ and $B^2=\langle0\rangle$, associativity gives
\[
T^2=0.
\]
Moreover,
\[
B=A^2=\operatorname{span}_F\{x^2,T(B)\}.
\]
Therefore
\[
\operatorname{rank}T=1
\]
and
\[
x^2\notin\operatorname{Im}T.
\]
Since a non-zero nilpotent operator on a two-dimensional vector space
satisfies
\[
\operatorname{Im}T=\ker T,
\]
we have
\[
T(x^2)\neq0.
\]
Put
\[
y=x^2,\qquad
z=xy.
\]
Then $\{x,y,z\}$ is a basis of $A$ and
\[
x^2=y,\qquad
xy=z,
\]
with all remaining products zero. This is $P_{3,5}$.

Finally, suppose
\[
A^2=A.
\]
By Lemma~\ref{lem:square-equals-algebra}, the algebra $A$ has an
identity element.

Suppose first that $A$ contains a non-trivial idempotent $e$. Then
\[
A=eA\oplus(1-e)A.
\]
The dimensions of these two ideals are $1$ and $2$. Using the
two-dimensional unital classification, we obtain
\[
F^3,\qquad
F\times F[\varepsilon]/(\varepsilon^2),
\qquad
F\times E,
\]
where $E/F$ is a quadratic field extension. These are
$P_{3,10}$, $P_{3,11}$ and $P_{3,12}(E)$.

Suppose now that $A$ has no non-trivial idempotents. Since $A$ is a
finite-dimensional commutative unital algebra over $F$, it is an
Artinian algebra. By the standard decomposition theorem for
commutative Artinian algebras,
\[
A\cong A_1\times\cdots\times A_r,
\]
where each $A_i$ is a local Artinian algebra. If $r\geq2$, then the
element
\[
(1,0,\ldots,0)
\]
is a non-trivial idempotent of $A$, contrary to our assumption.
Therefore $r=1$, and hence $A$ is local.

Let
\[
N=\operatorname{Nil}(A).
\]
For a commutative Artinian algebra, the nilradical coincides with the
Jacobson radical. Since $A$ is local, this radical is its unique
maximal ideal. Consequently,
\[
A/N
\]
is a field.

If
\[
N=\langle0\rangle,
\]
then $A$ itself is a field of dimension $3$ over $F$. Hence
\[
A=K,
\qquad
[K:F]=3,
\]
which gives $P_{3,13}(K)$.

The case
\[
\dim_F N=1
\]
cannot occur. Indeed, let
\[
N=Fn.
\]
Since $N$ is nilpotent,
\[
n^2=0.
\]
Multiplication on $N$ induces a unital $F$-algebra homomorphism
\[
A/N\longrightarrow\operatorname{End}_F(N)\cong F.
\]
But in this case $A/N$ would be a quadratic field extension of $F$.
A unital homomorphism from a field is injective, which is impossible
because a two-dimensional extension of $F$ cannot embed into $F$ as an
$F$-algebra.

Thus
\[
\dim_F N=2
\]
and
\[
A/N\cong F.
\]

If
\[
N^2=\langle0\rangle,
\]
we obtain $P_{3,14}$.

Assume
\[
N^2\neq\langle0\rangle.
\]
Then
\[
\dim_F N^2=1.
\]
We first show that there exists $u\in N$ such that
\[
u^2\neq0.
\]
Suppose, on the contrary, that every square in $N$ is zero, and let
$\{u,v\}$ be a basis of $N$. Write
\[
uv=\alpha u+\beta v.
\]
Associativity gives
\[
u(uv)=(u^2)v=0
\]
and
\[
v(uv)=u(v^2)=0.
\]
Hence
\[
\beta uv=0,\qquad
\alpha uv=0.
\]
If $uv\neq0$, then $\alpha=\beta=0$, a contradiction. Thus
\[
uv=0,
\]
which implies
\[
N^2=\langle0\rangle,
\]
again a contradiction.

Therefore there exists $u\in N$ with $u^2\neq0$. Put
\[
v=u^2.
\]
Since $N$ is nilpotent and $\dim_F N^2=1$, we have
\[
N^3=\langle0\rangle.
\]
Consequently
\[
uv=v^2=0,
\]
and $A$ is $P_{3,15}$.

It remains to determine the equivalence in the family $P_{3,4}(M)$.
For these algebras,
\[
\operatorname{Ann}(A)=Fz.
\]
Therefore every algebra isomorphism induces an invertible linear
transformation on
\[
A/Fz
\]
and sends $z$ to a non-zero scalar multiple of the corresponding
generator. It follows that
\[
M'=\lambda C^TMC
\]
for some
\[
C\in GL_2(F),\qquad
\lambda\in F^\times.
\]
Conversely, every such transformation determines an algebra
isomorphism.

For $P_{3,9}(E)$ we have
\[
\operatorname{Ann}(P_{3,9}(E)(+,\cdot))=Fz
\]
and hence
\[
P_{3,9}(E)/
\operatorname{Ann}(P_{3,9}(E)(+,\cdot))
\cong E.
\]
Thus two algebras of this type are isomorphic precisely when the
corresponding quadratic extensions are isomorphic as $F$-algebras.

For $P_{3,12}(E)=F\times E$, the quadratic field $E$ is the
two-dimensional simple ideal, so its $F$-algebra isomorphism type is
also preserved. For $P_{3,13}(K)$, the statement is immediate, since
the whole associative algebra is the field $K$.
\end{proof}

For the algebras in Theorem~\ref{thm:3d-trivial-lie}, the associated
Lie algebra is abelian and therefore
\[
\zeta(P(+,[\, ,\,]))=P.
\]
Their associative squares are grouped as follows:
\[
\begin{array}{c|c}
\dim_F P^2 & \text{algebras}\\
\hline
0 & P_{3,1}\\
1 & P_{3,2},\,P_{3,3},\,P_{3,4}(M)\\
2 & P_{3,5},\,P_{3,6},\,P_{3,7},\,P_{3,8},\,P_{3,9}(E)\\
3 & P_{3,10},\,P_{3,11},\,P_{3,12}(E),\,P_{3,13}(K),
    \,P_{3,14},\,P_{3,15}.
\end{array}
\]
In particular,
\[
\operatorname{Ann}(P(+,\cdot))=\langle0\rangle
\]
for $P_{3,10},\ldots,P_{3,15}$, since these algebras are unital.

\subsection{One-dimensional derived Lie algebra}

We now suppose
\[
\dim_F[P,P]=1.
\]
There are two essentially different situations according as the
derived algebra is central or non-central.

\subsubsection*{The central case}

Suppose first that
\[
[P,P]\subseteq\zeta(P(+,[\, ,\,])).
\]
A basis $\{x,y,z\}$ may then be chosen so that
\[
[x,y]=z
\]
and all other basic Lie brackets vanish.

\begin{theorem}
\label{thm:3d-central-derived}
Let $P$ be a three-dimensional Poisson algebra such that
\[
\dim_F[P,P]=1,
\qquad
[P,P]\subseteq\zeta(P(+,[\, ,\,])).
\]
Then $P$ has a basis $\{x,y,z\}$ in which
\[
[x,y]=z
\]
and
\[
x^2=\alpha z,\qquad
xy=\beta z,\qquad
y^2=\gamma z,
\]
while
\[
xz=yz=z^2=0.
\]
We denote this algebra by $P_{3,16}(M)$, where
\[
M=
\begin{pmatrix}
\alpha&\beta\\
\beta&\gamma
\end{pmatrix}.
\]

Two such algebras $P_{3,16}(M)$ and $P_{3,16}(M')$ are isomorphic if
and only if
\[
M'=\frac{1}{\det C}\,C^TMC
\]
for some
\[
C\in GL_2(F).
\]
\end{theorem}

\begin{proof}
Choose a basis $\{x,y,z\}$ with
\[
[x,y]=z,\qquad
[x,z]=[y,z]=0.
\]
Applying the Leibniz identity to the products of the basis elements
shows that the associative multiplication must have the form
\[
\begin{aligned}
x^2&=2\lambda x+\alpha z,\\
xy&=\mu x+\lambda y+\beta z,\\
y^2&=2\mu y+\gamma z,\\
xz&=\lambda z,\qquad
yz=\mu z,\qquad
z^2=0.
\end{aligned}
\]

Associativity gives
\[
(x^2)z=x(xz),
\]
and therefore
\[
2\lambda^2z=\lambda^2z.
\]
Thus
\[
\lambda=0.
\]
Similarly,
\[
(y^2)z=y(yz)
\]
gives
\[
\mu=0.
\]
Hence
\[
x^2=\alpha z,\qquad
xy=\beta z,\qquad
y^2=\gamma z,
\qquad
xz=yz=z^2=0.
\]

Conversely, every associative multiplication of this form is
commutative, associative and compatible with the given Lie bracket.

It remains to determine the isomorphisms. Put
\[
U=(x,y).
\]
Let
\[
\widetilde U=UC+z\ell,
\qquad
C\in GL_2(F),
\]
where $\ell$ is a row vector over $F$. Since
\[
[x,y]=z,
\]
we have
\[
[\widetilde x,\widetilde y]
=
(\det C)z.
\]
Thus, in order to retain the normalized relation
\[
[\widetilde x,\widetilde y]=\widetilde z,
\]
we must put
\[
\widetilde z=(\det C)z.
\]

Since
\[
zP=\langle0\rangle,
\]
the terms involving $z\ell$ do not affect the associative products.
Therefore the matrix of the associative multiplication in the new
basis is
\[
\widetilde M
=
\frac{1}{\det C}\,C^TMC.
\]
Hence two algebras $P_{3,16}(M)$ and $P_{3,16}(M')$ are isomorphic
if and only if
\[
M'
=
\frac{1}{\det C}\,C^TMC
\]
for some $C\in GL_2(F)$.

Equivalently, if
\[
\varphi:P_{3,16}(M)\longrightarrow P_{3,16}(M')
\]
is an isomorphism and $C$ is the matrix induced by $\varphi$ on
$P/Fz$, then
\[
\varphi(z)=(\det C)z'
\]
and direct comparison of the products gives
\[
(\det C)M=C^TM'C.
\]
Replacing $C$ by $C^{-1}$ yields the preceding equivalence relation.
Conversely, every matrix $C\in GL_2(F)$ satisfying this relation
determines a Poisson algebra isomorphism.
\end{proof}

For every $P_{3,16}(M)$,
\[
[P,P]=Fz
=
\zeta(P(+,[\, ,\,])),
\]
and the associated Lie algebra is nilpotent of class $2$.

Moreover,
\[
P^2=
\begin{cases}
\langle0\rangle, & M=0,\\
Fz, & M\neq0,
\end{cases}
\qquad
P^3=\langle0\rangle,
\]
and
\[
\dim_F\operatorname{Ann}(P(+,\cdot))
=
3-\operatorname{rank}M.
\]

\begin{corollary}
\label{cor:3d-central-derived-char-not-two}
Suppose that $\operatorname{char}(F)\neq2$. Then every non-zero
algebra $P_{3,16}(M)$ is isomorphic to one for which
\[
M=
\begin{pmatrix}
1&0\\
0&d
\end{pmatrix},
\qquad
d=\det M.
\]
Consequently, every non-zero algebra in this family can be written in
the form
\[
[x,y]=z,\qquad
x^2=z,\qquad
y^2=dz,
\]
where all remaining associative products are zero.
\end{corollary}

\begin{proof}
Since $\operatorname{char}(F)\neq2$ and $M$ is a non-zero symmetric
matrix, there exists a vector $v$ such that the corresponding
quadratic value is non-zero. Indeed, if all quadratic values were
zero, polarization would give
\[
2B(u,v)=0
\]
for all $u,v$, and hence $B=0$, contrary to $M\neq0$.

After a suitable change of basis, we may therefore assume that
\[
M=
\begin{pmatrix}
a&b\\
b&c
\end{pmatrix},
\qquad
a\neq0.
\]
Put
\[
d=\det M.
\]
For
\[
S=
\begin{pmatrix}
1&-b/a\\
0&1
\end{pmatrix},
\qquad
\det S=1,
\]
we have
\[
S^TMS
=
\begin{pmatrix}
a&0\\
0&d/a
\end{pmatrix}.
\]
Now put
\[
T=
\begin{pmatrix}
1&0\\
0&a
\end{pmatrix}.
\]
Then
\[
\frac{1}{\det T}\,
T^T
\begin{pmatrix}
a&0\\
0&d/a
\end{pmatrix}
T
=
\begin{pmatrix}
1&0\\
0&d
\end{pmatrix}.
\]
Hence the required normal form follows from
Theorem~\ref{thm:3d-central-derived}.
\end{proof}

Here $d=0$ corresponds to a non-zero matrix of rank $1$; the zero
matrix $M=0$ is a separate case.

In characteristic $2$, the matrix formulation of
Theorem~\ref{thm:3d-central-derived} is preferable, since it also
includes alternating symmetric forms, which need not be equivalent to
non-alternating ones.

\subsubsection*{The non-central case}

We next assume
\[
[P,P]\not\subseteq\zeta(P(+,[\, ,\,])).
\]
Then a basis $\{x,y,z\}$ can be chosen so that
\[
[x,y]=y,\qquad
z\in\zeta(P(+,[\, ,\,])).
\]

\begin{theorem}
\label{thm:3d-noncentral-derived}
Let $P$ be a three-dimensional Poisson algebra with
\[
\dim_F[P,P]=1,
\qquad
[P,P]\not\subseteq\zeta(P(+,[\, ,\,])).
\]
Then $P$ is isomorphic to exactly one of the following four algebras.
In all cases
\[
[x,y]=y,
\]
and all Lie brackets not determined by this equality are zero.

\begin{enumerate}
\item[\rm (i)]
$P_{3,17}$: the associative multiplication is zero.

\item[\rm (ii)]
$P_{3,18}$:
\[
x^2=z.
\]

\item[\rm (iii)]
$P_{3,19}$:
\[
z^2=z.
\]

\item[\rm (iv)]
$P_{3,20}$:
\[
z^2=z,\qquad
zx=x,\qquad
zy=y.
\]
\end{enumerate}

All associative products not explicitly displayed are zero.
\end{theorem}

\begin{proof}
Choose a basis $\{x,y,z\}$ with
\[
[x,y]=y,\qquad
[x,z]=[y,z]=0.
\]
The Leibniz identity forces the associative multiplication to have
the form
\[
\begin{aligned}
x^2&=2\alpha x+\beta z,\\
xy&=\alpha y,\\
xz&=\gamma x+\delta z,\\
y^2&=0,\\
yz&=\gamma y,\\
z^2&=\varepsilon z.
\end{aligned}
\]
Associativity yields
\[
\beta\gamma+\alpha^2=0,
\qquad
\gamma\delta=0,
\qquad
\gamma(\gamma-\varepsilon)=0,
\]
and
\[
\beta(\varepsilon-\gamma)
+2\alpha\delta-\delta^2=0.
\]

Suppose first that
\[
\gamma\neq0.
\]
Then
\[
\varepsilon=\gamma,\qquad
\delta=0,\qquad
\beta=-\frac{\alpha^2}{\gamma}.
\]
Put
\[
e=\gamma^{-1}z,
\qquad
u=x-\alpha e.
\]
Then
\[
e^2=e,\qquad
eu=u,\qquad
ey=y,
\]
while
\[
u^2=uy=y^2=0.
\]
Thus $e$ is the identity element of the associative algebra. Renaming
$u$ as $x$ and $e$ as $z$, we obtain $P_{3,20}$.

Suppose now that
\[
\gamma=0.
\]
Then
\[
\alpha=0
\]
and
\[
x^2=\beta z,\qquad
xz=\delta z,\qquad
z^2=\varepsilon z,
\]
where
\[
\beta\varepsilon=\delta^2.
\]

If
\[
\varepsilon\neq0,
\]
put
\[
e=\varepsilon^{-1}z
\]
and replace $x$ by
\[
x-\delta e.
\]
Then
\[
e^2=e
\]
and all other associative products vanish. Renaming $e$ as $z$, we
obtain $P_{3,19}$.

If
\[
\varepsilon=0,
\]
then
\[
\delta=0.
\]
When $\beta=0$, the associative multiplication is zero, giving
$P_{3,17}$.

When $\beta\neq0$, replacing $z$ by $\beta z$ gives
\[
x^2=z,
\]
which is $P_{3,18}$.

The four algebras are pairwise non-isomorphic. Indeed, their
associative structures have respectively
\[
\dim_F P^2=0,1,1,3,
\]
while in the two cases with $\dim_F P^2=1$ one has
\[
P^3=\langle0\rangle
\]
for $P_{3,18}$ and
\[
P^3=P^2
\]
for $P_{3,19}$.
\end{proof}

For $P_{3,17},\ldots,P_{3,20}$,
\[
[P,P]=Fy,
\qquad
\zeta(P(+,[\, ,\,]))=Fz.
\]
The associated Lie algebra is solvable but not nilpotent, since
\[
[[P,P],P]=Fy.
\]

The associative annihilators are
\[
\begin{array}{c|c}
P & \operatorname{Ann}(P(+,\cdot))\\
\hline
P_{3,17} & P\\
P_{3,18} & Fy\oplus Fz\\
P_{3,19} & Fx\oplus Fy\\
P_{3,20} & \langle0\rangle.
\end{array}
\]

\subsection{Two-dimensional derived Lie algebra}

We now consider
\[
\dim_F[P,P]=2.
\]
This case admits a particularly compact structural description.

\begin{theorem}
\label{thm:3d-derived-two}
Let $P$ be a three-dimensional Poisson algebra such that
\[
\dim_F[P,P]=2.
\]
Put
\[
V=[P,P].
\]
Then
\[
[V,V]=\langle0\rangle
\]
and there exists $x\notin V$ such that
\[
P=Fx\oplus V,
\qquad
[x,v]=A(v)
\quad (v\in V),
\]
where
\[
A\in GL(V).
\]

If
\[
\operatorname{char}(F)\neq2,
\]
the associative multiplication of $P$ is necessarily zero.

If
\[
\operatorname{char}(F)=2,
\]
then either the associative multiplication is zero or $P$ has a basis
$\{x,u,v\}$ such that
\[
[x,u]=u,\qquad
[x,v]=v+qu,
\]
and
\[
xv=u
\]
for some $q\in F$, all remaining associative products being zero.
\end{theorem}

\begin{proof}
Put
\[
V=[P,P].
\]
We first prove that $V$ is abelian.

Suppose otherwise. Since $\dim_F V=2$, choose a basis $\{u,v\}$ of
$V$ such that
\[
[u,v]=v.
\]
Let
\[
x\notin V
\]
and write
\[
[x,u]=au+bv,
\qquad
[x,v]=cu+dv.
\]
The Jacobi identity for $x,u,v$ gives
\[
c=0,\qquad
a=0.
\]
It follows that every basic Lie bracket belongs to $Fv$, contrary to
\[
\dim_F[P,P]=2.
\]
Therefore
\[
[V,V]=\langle0\rangle.
\]

Since
\[
[P,P]=[x,V],
\]
the operator
\[
A=\operatorname{ad}_x|_V
\]
is invertible.

Let $r,s,t\in V$. Since $V$ is abelian, the Leibniz identity gives
\[
[rs,t]
=
r[s,t]+s[r,t]
=
0.
\]
Thus every product $rs$ centralizes $V$. Since $A$ is invertible, the
centralizer of $V$ in the associated Lie algebra is exactly $V$.
Hence
\[
V^2\subseteq V.
\]

Write
\[
xr=\lambda(r)x+B(r),
\qquad r\in V,
\]
where
\[
\lambda:V\longrightarrow F
\]
and
\[
B:V\longrightarrow V
\]
are linear.

For $r,s\in V$, the Leibniz identity gives
\[
[xr,s]=r[x,s].
\]
Therefore
\[
\lambda(r)A(s)=rA(s).
\]
Since
\[
A(V)=V,
\]
we obtain
\[
rt=\lambda(r)t
\qquad
(r,t\in V).
\]
By commutativity,
\[
\lambda(r)t=\lambda(t)r.
\]
Since $\dim_F V=2$, this is possible for all $r,t\in V$ only when
\[
\lambda=0.
\]
Consequently
\[
V^2=\langle0\rangle,
\qquad
xV\subseteq V.
\]

Write
\[
xv=B(v)
\qquad(v\in V)
\]
and
\[
x^2=ax+w,
\qquad
w\in V.
\]
The Leibniz identity gives
\[
[x^2,x]=0.
\]
Hence
\[
A(w)=0.
\]
Since $A$ is invertible,
\[
w=0,
\]
and therefore
\[
x^2=ax.
\]

Applying the Leibniz identity and associativity to these products gives
\[
2B=aI,
\qquad
AB=BA,
\qquad
B^2=aB.
\]

Suppose first that
\[
\operatorname{char}(F)\neq2.
\]
Then
\[
B=\frac{a}{2}I.
\]
Substituting this into
\[
B^2=aB
\]
gives
\[
\frac{a^2}{4}I
=
\frac{a^2}{2}I,
\]
and therefore
\[
a=0.
\]
Thus
\[
B=0,
\]
and the associative multiplication is zero.

Let now
\[
\operatorname{char}(F)=2.
\]
Then
\[
2B=aI
\]
gives
\[
a=0,
\]
and hence
\[
B^2=0,\qquad
AB=BA.
\]
If
\[
B=0,
\]
the associative multiplication is again zero.

Suppose
\[
B\neq0.
\]
Since $\dim_F V=2$ and $B^2=0$, choose a basis $\{u,v\}$ of $V$ such
that
\[
B(u)=0,\qquad
B(v)=u.
\]
The relation
\[
AB=BA
\]
then implies that, in this basis,
\[
A=
\begin{pmatrix}
p&r\\
0&p
\end{pmatrix},
\qquad
p\neq0.
\]

Replacing $x$ by $p^{-1}x$ and simultaneously replacing $u$ by
$p^{-1}u$, we preserve the equality
\[
xv=u
\]
and normalize the diagonal coefficient of $A$ to $1$. Thus
\[
[x,u]=u,\qquad
[x,v]=v+qu
\]
for some $q\in F$, and
\[
xv=u.
\]

Conversely, for every $q\in F$ these operations define a Poisson
algebra: the associative multiplication is commutative and
associative, and the Leibniz identity is verified directly on the
basis elements.
\end{proof}

When the associative multiplication is zero, we denote the resulting
algebra by
\[
P_{3,21}(A),
\qquad
A\in GL_2(F),
\]
where
\[
[x,v]=A(v),
\qquad
v\in V.
\]

Two such algebras are isomorphic precisely when the corresponding
operators differ by similarity and multiplication by a non-zero
scalar. Equivalently,
\[
P_{3,21}(A)\cong P_{3,21}(A')
\]
if and only if
\[
A'=\lambda SAS^{-1}
\]
for some
\[
S\in GL_2(F),
\qquad
\lambda\in F^\times.
\]

In characteristic $2$, we denote the additional family by
\[
P_{3,22}(q),
\qquad
q\in F.
\]
Thus
\[
[x,u]=u,\qquad
[x,v]=v+qu,\qquad
xv=u.
\]

Moreover,
\[
P_{3,22}(q)\cong P_{3,22}(q')
\quad\Longleftrightarrow\quad
q=q'.
\]

Indeed, suppose
\[
\varphi:
P_{3,22}(q)\longrightarrow P_{3,22}(q')
\]
is an isomorphism. The subspaces
\[
P^2=Fu
\]
and
\[
[P,P]=Fu\oplus Fv
\]
are invariant under every Poisson algebra isomorphism. Hence
\[
\varphi(u)=a u',
\]
\[
\varphi(v)=b u'+c v',
\]
and
\[
\varphi(x)=d x'+r u'+s v',
\]
where
\[
a,c,d\neq0.
\]

The equality
\[
\varphi(xv)=\varphi(x)\varphi(v)
\]
gives
\[
a=dc.
\]
The equality
\[
\varphi([x,u])
=
[\varphi(x),\varphi(u)]
\]
gives
\[
d=1.
\]
Finally,
\[
\varphi([x,v])
=
[\varphi(x),\varphi(v)]
\]
gives
\[
q=q'.
\]

For all algebras considered in this subsection,
\[
[P,P]=V,
\qquad
\zeta(P(+,[\, ,\,]))=\langle0\rangle,
\]
and the associated Lie algebra is solvable but not nilpotent.

For $P_{3,22}(q)$,
\[
P^2=Fu,\qquad
P^3=\langle0\rangle,
\qquad
\operatorname{Ann}(P(+,\cdot))=Fu.
\]

\subsection{The perfect Lie case}

The remaining possibility for the Lie structure is
\[
[P,P]=P.
\]
In this case the Poisson compatibility leaves no freedom for the
associative multiplication.

\begin{theorem}
\label{thm:3d-perfect}
Let $P$ be a three-dimensional Poisson algebra such that
\[
[P,P]=P.
\]
Then the associated Lie algebra $P(+,[\, ,\,])$ is simple and
\[
P^2=\langle0\rangle.
\]
Conversely, every three-dimensional simple Lie algebra over $F$,
equipped with the zero associative multiplication, is a Poisson
algebra.
\end{theorem}

\begin{proof}
Put
\[
L=P(+,[\, ,\,]).
\]
Since
\[
[L,L]=L,
\]
every quotient of $L$ is perfect. A non-zero Lie algebra of dimension
one or two cannot be perfect. Hence $L$ has no non-zero proper ideals,
and therefore $L$ is simple.

By Lemma~\ref{lem:basic-poisson-ideals}, $P^2$ is a Lie ideal of $L$.
Therefore
\[
P^2=\langle0\rangle
\qquad\text{or}\qquad
P^2=P.
\]

Suppose
\[
P^2=P.
\]
By Lemma~\ref{lem:square-equals-algebra}, the associated commutative
associative algebra has an identity element $e$. By
Lemma~\ref{lem:idempotents-central},
\[
e\in\zeta(L).
\]
But a simple non-abelian Lie algebra has zero center, a contradiction.
Hence
\[
P^2=\langle0\rangle.
\]

The converse is immediate, since the Leibniz identity is automatically
satisfied when the associative multiplication is zero.
\end{proof}

We denote these algebras by
\[
P_{3,23}(\mathfrak{s}),
\]
where $\mathfrak{s}$ runs through the isomorphism classes of
three-dimensional simple Lie algebras over $F$. Thus
\[
P_{3,23}(\mathfrak{s})(+,[\, ,\,])
\cong\mathfrak{s},
\qquad
P_{3,23}(\mathfrak{s})^2=\langle0\rangle.
\]

\subsection{The complete three-dimensional classification}

We can now collect the preceding results. The notation
$P_{3,1},\ldots,P_{3,23}$ refers to structural entries in the
classification rather than to twenty-three individual isomorphism
classes: several of these entries are parameterized families whose
number of isomorphism classes depends on the ground field.

\begin{theorem}
\label{thm:3d-classification}
Every three-dimensional Poisson algebra over an arbitrary field $F$
is isomorphic to one of the algebras
\[
P_{3,1},\ldots,P_{3,15},
\qquad
P_{3,16}(M),
\qquad
P_{3,17},\ldots,P_{3,20},
\]
\[
P_{3,21}(A),
\qquad
P_{3,22}(q)
\quad\text{if }\operatorname{char}(F)=2,
\qquad
P_{3,23}(\mathfrak{s}),
\]
described above.

The parameters are subject precisely to the following
identifications:
\[
P_{3,4}(M)\cong P_{3,4}(M')
\quad\Longleftrightarrow\quad
M'=\lambda C^TMC
\]
for some
\[
C\in GL_2(F),
\qquad
\lambda\in F^\times;
\]
\[
P_{3,16}(M)\cong P_{3,16}(M')
\quad\Longleftrightarrow\quad
M'=\frac{1}{\det C}C^TMC
\]
for some
\[
C\in GL_2(F);
\]
\[
P_{3,21}(A)\cong P_{3,21}(A')
\]
if and only if
\[
A'=\lambda SAS^{-1}
\]
for some
\[
S\in GL_2(F),
\qquad
\lambda\in F^\times;
\]
and
\[
P_{3,22}(q)\cong P_{3,22}(q')
\quad\Longleftrightarrow\quad
q=q'.
\]

The field extensions occurring in $P_{3,9}(E)$,
$P_{3,12}(E)$ and $P_{3,13}(K)$ are taken up to
$F$-algebra isomorphism. In the family
$P_{3,23}(\mathfrak{s})$, the parameter $\mathfrak{s}$ runs through
the $F$-isomorphism classes of three-dimensional simple Lie algebras
over $F$. Thus
\[
P_{3,23}(\mathfrak{s})
\cong
P_{3,23}(\mathfrak{s}')
\quad\Longleftrightarrow\quad
\mathfrak{s}\cong_F\mathfrak{s}'.
\]
\end{theorem}

\begin{proof}
The dimension of the derived Lie algebra is one of
\[
0,\quad1,\quad2,\quad3.
\]

If
\[
[P,P]=\langle0\rangle,
\]
Theorem~\ref{thm:3d-trivial-lie} applies.

If
\[
\dim_F[P,P]=1,
\]
the derived algebra is either contained in the Lie center or it is not.
These two cases are covered by
Theorems~\ref{thm:3d-central-derived} and
\ref{thm:3d-noncentral-derived}.

If
\[
\dim_F[P,P]=2,
\]
Theorem~\ref{thm:3d-derived-two} applies.

Finally, if
\[
\dim_F[P,P]=3,
\]
then
\[
[P,P]=P,
\]
and Theorem~\ref{thm:3d-perfect} applies.

The different main groups cannot be isomorphic because
\[
\dim_F[P,P]
\]
is an isomorphism invariant. Within the case
$\dim_F[P,P]=1$, centrality of $[P,P]$ separates the two families.

When
\[
[P,P]=\langle0\rangle,
\]
the invariant
\[
\dim_F P^2
\]
first separates the four groups occurring in
Theorem~\ref{thm:3d-trivial-lie}. Within these groups, the associative
powers, the dimension of the associative annihilator, the presence of
idempotents and the direct-sum decomposition distinguish the
individual structural types. The remaining identifications are
exactly the parameter identifications stated above.
\end{proof}

The classification shows a qualitative difference between dimensions
two and three. In dimension two, one of the two Poisson
multiplications must vanish. In dimension three, there already exist
several families for which both the associative and Lie
multiplications are non-zero.

The dependence on the ground field is also substantially stronger in
dimension three. Quadratic extensions occur in $P_{3,9}(E)$ and
$P_{3,12}(E)$, cubic extensions occur in $P_{3,13}(K)$, and the
isomorphism classes inside the matrix families $P_{3,4}(M)$ and
$P_{3,16}(M)$ depend on the arithmetic of $F$. The Lie structures
arising when $\dim_F[P,P]=2$ and when $[P,P]=P$ may also depend on
the ground field. In addition, characteristic $2$ gives the extra
family $P_{3,22}(q)$.

\subsection*{Relation with the nilpotent classification}

We next compare our results with the classification of nilpotent
Poisson algebras obtained in
\cite{AbdelwahabBarreiroCalderonFernandez2023}.
The notion of nilpotency used there involves both Poisson
multiplications. For the purpose of this comparison, put
\[
\mathcal N_0(P)=P,\qquad
\mathcal N_{n+1}(P)
=
\mathcal N_n(P)P+[\mathcal N_n(P),P].
\]
Thus $P$ is nilpotent in this sense if
$\mathcal N_m(P)=\langle0\rangle$ for some $m$.
The notation $\mathcal N_n(P)$ is used here in order to distinguish
this series from the associative powers $P^n$ used throughout the
paper.

\begin{proposition}
\label{prop:nilpotent-specialization}
Assume that $\operatorname{char}(F)\neq2$.

\begin{itemize}

\item[(i)]
Among the two-dimensional Poisson algebras of
Theorem~\ref{thm:dimension-two}, the nilpotent ones are precisely
\[
P_{2,1}
\qquad\text{and}\qquad
P_{2,2}.
\]

\item[(ii)]
Among the three-dimensional Poisson algebras of
Theorem~\ref{thm:3d-classification}, the nilpotent ones are precisely
\[
P_{3,1},\qquad
P_{3,3},\qquad
P_{3,4}(M),\qquad
P_{3,5},\qquad
P_{3,16}(M).
\]

\end{itemize}

Thus, in characteristic different from $2$, the nilpotent members of
our two- and three-dimensional classifications are exactly the
algebras listed above, in agreement with the previously known
low-dimensional nilpotent theory developed in
\cite{AbdelwahabBarreiroCalderonFernandez2023}.
\end{proposition}

\begin{proof}
For $P_{2,1}$ we have
\[
\mathcal N_1(P)=\langle0\rangle,
\]
whereas for $P_{2,2}$,
\[
\mathcal N_1(P)=Fb,
\qquad
\mathcal N_2(P)=\langle0\rangle.
\]
Thus both algebras are nilpotent.

Each of $P_{2,3},\ldots,P_{2,6}$ contains a non-zero idempotent.
If $e$ is such an idempotent and
$e\in\mathcal N_n(P)$, then
\[
e=e^2\in\mathcal N_n(P)P
\subseteq\mathcal N_{n+1}(P).
\]
Hence none of these algebras is nilpotent. Finally, for $P_{2,7}$,
\[
\mathcal N_1(P)=Fb,
\qquad
[Fb,P]=Fb,
\]
so the series does not terminate.

We turn to dimension three. For $P_{3,1}$,
\[
\mathcal N_1(P)=\langle0\rangle.
\]
For $P_{3,3}$ and $P_{3,4}(M)$,
\[
\mathcal N_1(P)=Fz,
\qquad
\mathcal N_2(P)=\langle0\rangle.
\]
For $P_{3,5}$,
\[
\mathcal N_1(P)=Fy\oplus Fz,\qquad
\mathcal N_2(P)=Fz,\qquad
\mathcal N_3(P)=\langle0\rangle.
\]
Finally, for every $P_{3,16}(M)$ we have
\[
P^2\leq Fz,\qquad
[P,P]=Fz,\qquad
zP=\langle0\rangle,\qquad
[z,P]=\langle0\rangle.
\]
Therefore
\[
\mathcal N_1(P)=Fz,
\qquad
\mathcal N_2(P)=\langle0\rangle.
\]

It remains to exclude the other structural types.
Each of
\[
P_{3,2},\quad P_{3,6},\ldots,P_{3,15}
\]
contains a non-zero idempotent, and hence cannot be nilpotent by the
argument above. For $P_{3,17},\ldots,P_{3,20}$ we have
\[
[P,P]=Fy,
\qquad
[Fy,P]=Fy,
\]
so their lower central series cannot terminate. For $P_{3,21}(A)$,
putting $V=[P,P]$ gives
\[
[V,P]=V,
\]
since the operator defining the Lie multiplication on $V$ is
invertible. Finally, for $P_{3,23}(\mathfrak{s})$,
\[
[P,P]=P.
\]
Thus none of these algebras is nilpotent, completing the proof.
\end{proof}

\subsection*{Specialization to the complex field}

We conclude by comparing our classification with the known
classification of three-dimensional complex Poisson algebras obtained
in~\cite{AbdelwahabFernandezMartin2025}. When $F=\mathbb C$, the
quadratic and cubic field-extension cases disappear, the family
$P_{3,22}(q)$ does not occur, and the field-dependent matrix and Lie
families reduce to their complex forms.

For convenience, we denote the entries of Theorem~2.3 in
\cite{AbdelwahabFernandezMartin2025} by $Q_i$ or $Q_i^\lambda$,
according as they are non-parametric or parametric. Put
\[
H=
\begin{pmatrix}
0&1\\
1&0
\end{pmatrix},
\qquad
J_2=
\begin{pmatrix}
1&1\\
0&1
\end{pmatrix}.
\]
After suitable changes of basis, the correspondence is as follows:
\[
\begin{array}{c|l}
\text{Type in \cite{AbdelwahabFernandezMartin2025}}
&
\text{Corresponding type in the present paper}
\\
\hline
Q_1
& P_{3,1}
\\
Q_2
& P_{3,16}(0)
\\
Q_3
& P_{3,21}(J_2)
\\
Q_4^{0}
& P_{3,17}
\\
Q_4^{\lambda},\ \lambda\neq0
& P_{3,21}(\operatorname{diag}(1,\lambda))
\\
Q_5
& P_{3,23}(\mathfrak{sl}_2(\mathbb C))
\\
Q_6
& P_{3,5}
\\
Q_7
& P_{3,10}
\\
Q_8
& P_{3,11}
\\
Q_9
& P_{3,15}
\\
Q_{10}
& P_{3,7}
\\
Q_{11}
& P_{3,8}
\\
Q_{12}
& P_{3,6}
\\
Q_{13}
& P_{3,3}
\\
Q_{14}
& P_{3,18}
\\
Q_{15}
& P_{3,16}(\operatorname{diag}(1,0))
\\
Q_{16}^{0}
& P_{3,4}(H)
\\
Q_{16}^{\lambda},\ \lambda\neq0
& P_{3,16}(\lambda^{-1}H)
\\
Q_{17}
& P_{3,14}
\\
Q_{18}
& P_{3,20}
\\
Q_{19}
& P_{3,2}
\\
Q_{20}
& P_{3,19}
\end{array}
\]

Let us briefly explain the two parametric correspondences. For the
family $Q_4^\lambda$ with $\lambda\neq0$, the projective-similarity
relation in $P_{3,21}(A)$ gives precisely the identification
\[
\lambda\sim\lambda^{-1}.
\]
The parameter value $\lambda=0$ corresponds to the separate type
$P_{3,17}$.

For $Q_{16}^{\lambda}$ with $\lambda\neq0$, the multiplication and
Lie bracket may be written in the form
\[
e_1e_2=e_3,
\qquad
[e_1,e_2]=\lambda e_3.
\]
Putting
\[
x=e_1,\qquad
y=e_2,\qquad
z=\lambda e_3,
\]
we obtain
\[
[x,y]=z,
\qquad
xy=\lambda^{-1}z.
\]
Thus this algebra corresponds to
\[
P_{3,16}(\lambda^{-1}H),
\]
and the determinant parameter of
Corollary~\ref{cor:3d-central-derived-char-not-two} is
\[
d=\det(\lambda^{-1}H)=-\lambda^{-2}.
\]
Accordingly, the identification $\lambda\sim-\lambda$ in the complex
classification corresponds precisely to equality of the determinant
parameter $d$.

Finally, since $\mathbb C$ has no non-trivial finite field extensions
and every non-degenerate symmetric bilinear form of dimension $2$ over
$\mathbb C$ is equivalent up to scalar congruence, no additional
complex isomorphism classes arise from
$P_{3,4}(M)$, $P_{3,9}(E)$, $P_{3,12}(E)$ or $P_{3,13}(K)$.
Consequently, the classification obtained in this paper specializes,
up to changes of basis and notation, exactly to the known
classification of three-dimensional complex Poisson algebras
in~\cite{AbdelwahabFernandezMartin2025}.

\end{document}